\documentclass[uplatex]{amsart}
\usepackage[top=30truemm,bottom=30truemm,left=25truemm,right=25truemm]{geometry}

\usepackage{amsmath, amsthm, amssymb, mathtools, mathrsfs}
\usepackage{enumerate}
\usepackage[all]{xy}
\usepackage{array, booktabs}
\usepackage{float}
\usepackage{extarrows}
\usepackage{tabularx}
\usepackage{graphics}

\newcommand{\iu}{\sqrt{-1}}

\newcommand{\Z}{\mathbb{Z}}
\newcommand{\Q}{\mathbb{Q}}
\newcommand{\R}{\mathbb{R}}
\newcommand{\C}{\mathbb{C}}
\newcommand{\bbH}{\mathbb{H}}
\newcommand{\bbP}{\mathbb{P}}

\newcommand{\veps}{\varepsilon}

\DeclareMathOperator{\ImNew}{Im}

\DeclareMathOperator{\SL}{SL}

\DeclareMathOperator{\GL}{GL}

\def\i<#1>{\langle #1 \rangle}
\def\l<#1>{\left\langle #1 \right\rangle}
\makeatletter
\renewcommand{\theenumi}{{\upshape (\@roman\c@enumi)}}

\renewcommand{\p@enumii}{}
\renewcommand{\theenumii}{{\upshape (\@alph\c@enumii)}}

\makeatother

\theoremstyle{plain}
\newtheorem{theorem}{Theorem}[section]
\newtheorem{definition}[theorem]{Definition}
\newtheorem{remark}[theorem]{Remark}
\newtheorem{proposition}[theorem]{Proposition}
\newtheorem{lemma}[theorem]{Lemma}

\newtheorem*{notation*}{Notation}

\begin{document}

\title{A continuity of cycle integrals of modular functions}

\author{Yuya Murakami}
\date{\today}
\thanks{Mathematical Institute, Tohoku University, 6-3, Aoba, Aramaki, Aoba-Ku, Sendai 980-8578, JAPAN
		\textit{E-mail address}: \texttt{yuya.murakami.s8@dc.tohoku.ac.jp}} 

\maketitle

\begin{abstract}
	In this paper we study a continuity of the ``values" of modular functions at the real quadratic numbers which are defined in terms of their cycle integrals along the associated closed geodesics. 
	Our main theorem reveals a more finer structure of the continuity of these values with respect to continued fraction expansions and it turns out that it is different from the continuity with respect to Euclidean topology.
\end{abstract}


\section{Introduction} \label{sec:intro}


Let 
$ \bbH := \{ z=x+y\iu \in \C \mid x, y \in \R, y>0 \} $
be the upper half plane and
$ f \colon \bbH \to \C $ be a holomorphic modular function with respect to $ \SL_2(\Z) $.
Kaneko \cite{Kan} introduced  the ``value" of $ f $ at any real quadratic number $ w $ as follows.
For such a $ w $ and its non-trivial Galois conjugate $ w' $ we define
\[
\eta_w := \bigg( \dfrac{1}{z - w'} - \dfrac{1}{z - w} \bigg) dz.
\]
Let $ \SL_2(\Z)_w $ be the stabilizer of $ w $ and
$ \gamma_w = \begin{pmatrix} * & * \\ c & d \end{pmatrix} $ be the unique element of $ \SL_2(\Z)_w $ such that 
$ \SL_2(\Z)_w = \{ \pm \gamma_w^n \mid n \in \Z \}, \veps_w := j(\gamma_w, w) := cw + d > 1 $.
For $ z_0 \in  \bbH $ we define
\[
\widetilde{f}(w) := \int_{z_0}^{\gamma_w z_0} f \eta_w, \  \widetilde{1}(w) := \int_{z_0}^{\gamma_w z_0} \eta_w.
\]
The integral is independent of $ z_0 $ since $ f \circ \gamma_w = f, \gamma_w^* \eta_w = \eta_w $ and $ f $ is holomorphic.
We set
\[
f(w) := \frac{\widetilde{f}(w)}{\widetilde{1}(w)}.
\]

This value has a representation as a cycle integral
\[
f(w) = \frac{1}{\mathrm{length}(\SL_2(\Z)_w \backslash S_{w})} \int_{\SL_2(\Z)_w \backslash S_{w}} f ds
\]
where $ ds := y^{-1} \sqrt{dx^2 + dy^2} $, $ S_w $ is the geodesic in $ \bbH $ from $ w' $ to $ w $ and
\[
\mathrm{length}(\SL_2(\Z)_w \backslash S_{w}) := \int_{\SL_2(\Z)_w \backslash S_{w}} ds
= \widetilde{1}(w).
\]

The values $ f(w) $ and  $ \widetilde{f}(w) $ have been studied in the last decade.
Duke, Imamo\={g}lu and T\'{o}th \cite{DIT} proved that a generating function of traces of these values is a mock modular form of weight $ 1/2 $ for $ \Gamma_0(4) $.
This work is a real quadratic analog of Zagier's work on traces of singular moduli in \cite{Zag}. 
Masri \cite{Mas} studied equidistribution of traces of these values.
P\"{a}pcke \cite{Pap} showed the interesting formula
\[
\lim_{n \to \infty} j_1\left( \frac{n+ \sqrt{n^2 - 4}}{2} \right) = 0.
\]
We can regard it as a certain continuity of $ f(w) $.
We remark that the real quadratic numbers in the left hand side have the continued fraction expansions
\[
\frac{n+ \sqrt{n^2 - 4}}{2}
= n + \cfrac{1}{n + {\cfrac{1}{n + \cfrac{1}{\ddots} } } }.
\]

Kaneko \cite{Kan} conjectured another continuity of $ f(w) $ when $ w $ is a Markov quadratic number and this conjecture is proved by Bengoechea and Imamo\={g}lu \cite{BI}.
In this paper, we pursue more finer structure in the continuity with respect to the continued fraction expansions.
Our main result is to show that the function $ f(w) $ has a continuity with respect to how many cyclic parts appear in a continued fraction expansion of $ w $.

Each real number can be represented by the unique continued fraction
\[
[k_1, k_2, k_3, \dots] := k_1 + \cfrac{1}{k_2 + {\cfrac{1}{k_3 + \cfrac{1}{\ddots} } } }
\]
where $ k_1 $ is an integer and $ k_2, k_3, \dots $ are positive integers.
Also we set a periodic continued fraction
\[
[k_1, \dots, k_{r}, \overline{k_{r+1}, \dots, k_s}] := [k_1, \dots, k_{r}, k_{r+1}, \dots, k_s, k_{r+1}, \dots, k_s, k_{r+1}, \dots, k_s, \dots].
\]
For a real number, being quadratic is equivalent to that it has a periodic continued fraction expansion.

For an even number $ r \in 2\Z_{\ge 0} $ and positive integers $ k_1, \dots, k_{r} $, we call $ W=(k_1, \dots, k_{r}) $ an even word.
For even words $ W_1 = (k_1^{(1)}, \dots, k_{r_1}^{(1)}), \dots, W_n = (k_1^{(n)}, \dots, k_{r_n}^{(n)}) $, define its product
\[
W_1 \cdots W_n :=(k_1^{(1)}, \dots, k_{r_1}^{(1)}, \dots, k_1^{(n)}, \dots, k_{r_n}^{(n)}).
\]
For an even word $ W = (k_1, \dots, k_{r}) $, let 
\[
[\overline{W}] := [\overline{k_1, \dots, k_r}],
\]
\[
N(W) := \max \{ n \in \Z_{>0} \mid  \exists \text{ an even word } W_1 \text{ s.t. } W = W_1^n \}.
\]
Here, if $ W = W_1^{N(W)} $, then $ N(W_1) = 1 $.
For example, $ N((1,2,1,2)) = 2, N((1, 2)) = 1 $ and $ N((1,1)) = 1 $.

Our main result is the following:

\begin{theorem} \label{thm:main}
	For sequences $\{ a_n^{(1)} \}_{n=1}^{\infty}, \dots, \{ a_n^{(k)} \}_{n=1}^{\infty}$ in $ \Z_{\ge 0} $ diverging to the positive infinity, let
	\[
	a^{(i)} := \lim_{n \to \infty} \frac{a_n^{(i)}}{a_n^{(1)} + \cdots + a_n^{(k)}}, \,
	1 \le i \le k.
	\]
	For even words $V_0, \dots, V_k $ and non-empty even words $ W_1, \dots, W_k$, let
	\[
	w_i := [\overline{W_i}], \  
	w^{(n)} := [\overline{V_0 W_1^{a_n^{(1)}} V_1 W_2^{a_n^{(2)}} V_2 \dots W_k^{a_n^{(k)}} V_k}].
	\]
	Then we have
	\[
	\lim_{n \to \infty} f(w^{(n)}) 
	= \frac{a^{(1)} N(W_1) \widetilde{f} (w_1) + \cdots + a^{(k)} N(W_k) \widetilde{f} (w_k)}{a^{(1)} N(W_1) \widetilde{1} (w_1) + \cdots + a^{(k)} N(W_k) \widetilde{1} (w_k)}.
	\]
\end{theorem}

\begin{remark}
	\begin{enumerate}
		\item For the case when $k=1, a_n^{(1)} = n$, our main theorem says that
		\[
		\lim_{n \to \infty} f([\overline{V_0 W_1^n V_1}]) = f(w_1).
		\]
		In particular, if $V_0, V_1, W_1$ is products of $(1,1), (2,2)$ and $V_0 = \emptyset$ or $V_1 = \emptyset$, then this result is conjectured in {\upshape \cite{Kan}} and proved in {\upshape \cite{BI}}.
		Here a sequence $ [\overline{V_0 W_1^n V_1}] $ converges to $ \gamma_{V_0}^{} w_1 = [V_0 \overline{W_1}] $ and its non-trivial Galois conjugate converges to $ w_1' $ in the sense of Euclidean metric by Lemma $ 1.28 $ in {\upshape \cite{Aig}}.
		\item For the case when $k=2, a_n^{(1)} = a_n^{(2)} = n $ and $ V_0 = V_1 = V_2 = \emptyset $, our main theorem says that
		\[
		\lim_{n \to \infty} f([\overline{W_1^n W_2^n}]) 
		= \frac{ \widetilde{f} (w_1) + \widetilde{f} (w_2)}{ \widetilde{1} (w_1) + \widetilde{1} (w_2) }.
		\]
		Here a sequence $ [\overline{W_1^n W_2^n}] $ converges to $ w_1 $ and its non-trivial Galois conjugate converges to $ w_2' $ in the sense of Euclidean metric by Lemma $ 1.28 $ in {\upshape \cite{Aig}}.
		Since $ \widetilde{f} (w_2) = \widetilde{f} (w_2') $ and $ \widetilde{1} (w_2) = \widetilde{1} (w_2') $ by a proof of Proposition  {\upshape (ii)} in {\upshape \cite{Kan}}, 
		we can reconstruct the limit of $  f([\overline{W_1^n W_2^n}])  $ from the limits of $ [\overline{W_1^n W_2^n}] $ and $ [\overline{W_1^n W_2^n}]' $ in the sense of Euclidean metric.
		\item For the case when $k=3, a_n^{(1)} = a_n^{(2)} = a_n^{(3)} = n $ and $ V_0 = V_1 = V_2 = V_3 = \emptyset $, our main theorem says that
		\[
		\lim_{n \to \infty} f([\overline{W_1^n W_2^n W_3^n}]) 
		= \frac{ \widetilde{f} (w_1) + \widetilde{f} (w_2) + \widetilde{f} (w_3)}{ \widetilde{1} (w_1) + \widetilde{1} (w_2) + \widetilde{1} (w_3) }.
		\]
		Here a sequence $ [\overline{W_1^n W_2^n W_3^n}] $ converges to $ w_1 $ and its non-trivial Galois conjugate converges to $ w_3' $ in the sense of Euclidean metric by Lemma $ 1.28 $ in {\upshape \cite{Aig}}.
		Thus we cannot reconstruct $ w_2 $ from $ [\overline{W_1^n W_2^n W_3^n}] $ by Euclidean metric 
		and the function $ f $ on the set of real quadratic numbers does not satisfy the continuity with respect to Euclidean metric.
	\end{enumerate}
\end{remark}

\begin{remark}
	We set the sign of $ \eta_w $ in the definition because of the fact $ \widetilde{1}(w) = 2 \log \veps_w $.
	We prove it in section $ \ref{sec:pre} $, although we do not need it to prove our main theorem. 
\end{remark}

%
%
%
%
%
%


\section{Preliminaries} \label{sec:pre}


In this section, we prepare several notation.

\begin{definition}
	For non-empty even words $ V= (k_1, \dots, k_r), W=(\ell_1, \dots, \ell_s) $, define a real quadratic number
	\[
	[V \overline{W}] := [k_1, \dots, k_r, \overline{\ell_1, \dots, \ell_s}].
	\]
\end{definition}

\begin{definition}
	For a non-empty even word $ W=(k_1, \dots, k_{r}) $, define
	\[
	\gamma_W^{} := \begin{pmatrix} k_1 & 1 \\ 1 & 0 \end{pmatrix} \begin{pmatrix} k_2 & 1 \\ 1 & 0 \end{pmatrix} \dots \begin{pmatrix} k_r & 1 \\ 1 & 0 \end{pmatrix}
	\in \SL_2(\Z).
	\]
	
	For the empty word $ \emptyset $, define
	\[
	\gamma_\emptyset^{} := \begin{pmatrix} 1 & 0 \\ 0 & 1 \end{pmatrix}.
	\]
\end{definition}

\begin{lemma} \label{lem:gamma_W}
	We have the following properties.
	\begin{enumerate}
		\item \label{item:lem:gamma_W;1} For even words $ W_1, \dots, W_n $, we have $ \gamma_{W_1 \cdots W_n }^{} = \gamma_{W_1}^{} \dots \gamma_{W_n}^{} $.
		\item \label{item:lem:gamma_W;2} For even words $ V= (k_1, \dots, k_r), W=(\ell_1, \dots, \ell_s) $ and a real quadratic number $ w= [\overline{W}] $,
		we have $ \gamma_V^{} w = [V \overline{W}] $.
		
		In particular, $ \gamma_W^{} w = w $.
		\item \label{item:lem:gamma_W;3} For a reduced real quadratic number $w$, there exists the minimal even word $ W $ such that $ w = [\overline{W}] $.
		Then we have $ N(W)=1 $.
		In this case, $ \gamma_W^{} = \gamma_{w} $.
		
		Generally, for a non-empty even word $ W $ and a real quadratic number $ w = [\overline{W}] $,
		we have $ \gamma_W^{} = \gamma_{w}^{N(W)} $.
		\item \label{item:lem:gamma_W;4} For a real quadratic number $w$, there exist even words $ V, W $ such that $ w = [V \overline{W}] $.
		Then we have $ \gamma_V^{} \gamma_W^{} \gamma_V^{-1} = \gamma_w^{N(W)} $.
	\end{enumerate}
\end{lemma}

\begin{proof}
	\ref{item:lem:gamma_W;1} and \ref{item:lem:gamma_W;2} follows by definition.
	For	\ref{item:lem:gamma_W;3}, it follows that $ N(W)=1 $ by the minimality of $ W $.
	We have $ \gamma_W^{} \in \SL_2(\Z)_{w} $ by \ref{item:lem:gamma_W;2} and $ j(\gamma_{W}, w) > 1 $ since $ w > 1 $ and all entries of $ \gamma_W^{} $ are positive.
	Thus $ \gamma_W^{} = \gamma_{w} $ by the minimality of $ W $.
	
	For \ref{item:lem:gamma_W;4}, since any real quadratic number has a periodic continued fraction expansion, one can find such $ V, W $ in the claim.
	Let $ w_0 := [\overline{W}] $.
	We have $ \gamma_V^{-1} w = w_0 $ by \ref{item:lem:gamma_W;1} and thus 
	$ \SL_2(\Z)_{w} = \gamma_V^{-1} \SL_2(\Z)_{w_0} \gamma_V^{} $.
	Here $ \gamma_{w}, \gamma_{w_0} $ are the unique generators of $ \SL_2(\Z)_{w}, \SL_2(\Z)_{w_0} $ such that
	$ j(\gamma_{w}, w) > 1, j(\gamma_{w_0}, w_0) > 1 $ respectively.
	Since
	\begin{align*}
		& j( \gamma_V^{-1} \gamma_{w_0} \gamma_V^{}, w) 
		= j( \gamma_V^{-1} \gamma_{w_0} \gamma_V^{}, \gamma_V^{-1} w_0)
		= j( \gamma_V^{-1}, \gamma_{w_0} \gamma_V^{} \gamma_V^{-1} w_0)
		j( \gamma_{w_0}, \gamma_V^{} \gamma_V^{-1} w_0)
		j( \gamma_V^{}, \gamma_V^{-1} w_0) \\
		&= j( \gamma_V^{-1}, w_0)
		j( \gamma_{w_0}, w_0)
		j( \gamma_V^{-1}, w_0)^{-1}
		= j( \gamma_{w_0}, w_0) >1,
	\end{align*}
	we have $  \gamma_V^{} \gamma_{w_0} \gamma_V^{-1} = \gamma_{w} $.
	By \ref{item:lem:gamma_W;3}, we obtain $ \gamma_V^{} \gamma_W^{} \gamma_V^{-1} = \gamma_w^{N(W)} $.
\end{proof}

\begin{definition}
	Set 
	\[
	\infty := \lim_{t \to + \infty} t \iu, \
	\bbP^1(\R) := \R \cup \{ \infty\}.
	\]
	For $ x', x \in \bbP^1(\R) $, define
	\[
	\eta_{x', x} (z) := \bigg( \dfrac{1}{z - x'} - \dfrac{1}{z - x} \bigg) dz.
	\]
	Here, let $ 1/(z - \infty) := 0 $ if $ x = \infty $ or $ x' = \infty $.
\end{definition}

The following lemma follows from a direct calculation and there details are omitted.

\begin{lemma} \label{lem:eta}
	For any $ x, x' \in \bbP^1(\R) $ and $\gamma \in \SL_2(\R)$, we have 
	$ \gamma^* \eta_{x', x} = \eta_{\gamma^{-1}x', \gamma^{-1}x}$.
	In particular, we have $\gamma^* \eta_w = \eta_{\gamma^{-1}w}$ for a real quadratic number $w$ and $\gamma \in \SL_2(\Z)$.
\end{lemma}

\begin{definition}
For a non-empty even word $ W $ with $ w = [\overline{W}] $ and $ z_0 \in \bbH $, define
\[
\eta_W^{} := \eta_w, \  \widetilde{f}(W) := \int_{z_0}^{\gamma_{W}^{} z_0} f \eta_{W}^{}.
\]
\end{definition}

\begin{lemma} \label{lem:tilde_f}
	For a non-empty even word $ W $ with $ w = [\overline{W}] $, we have
	\[
	\widetilde{f}(W) = N(W) \widetilde{f}(w).
	\]
	Thus $ \widetilde{f}(W) $ is independent of $ z_0 $	and we have
	\[
	f(w) = \frac{\widetilde{f}(W)}{\widetilde{1}(W)}.
	\]
\end{lemma}

\begin{proof}
	it follows that $ \gamma_W^{} = \gamma_{w}^{N(W)} $ by Lemma \ref{lem:gamma_W} \ref{item:lem:gamma_W;3},
	$ \gamma_w^* \eta_w = \eta_w $ by Lemma \ref{lem:eta} and $ f \circ \gamma_w = f $.
	Thus we have $ \widetilde{f}(W) = N(W) \widetilde{f}(w) $.
\end{proof}

For a real quadratic number $w$, there exist integers $ a, b, c $ such that
\[
aw^2 + bw + c = 0, \ \mathrm{gcd}(a,b,c) = 1, \ w = \frac{-b + \sqrt{b^2 - 4ac}}{2a}.
\]
Let $ D := b^2 - 4ac > 0 $ and $ K := \Q(\sqrt{D}) $.
We have a group isomorphism
\[
\begin{array}{ccccc}
\GL_2(\Z)_w & \xrightarrow{\sim} & \Z[aw]^\times \\
\begin{pmatrix} p & q \\ r & s \end{pmatrix}
& \longmapsto & rw+s \\
\begin{pmatrix} \frac{x - by}{2} & -cy \\ ay & \frac{x + by}{2} \end{pmatrix}
& \reflectbox{\ensuremath{\longmapsto}} & \frac{x + y \sqrt{D}}{2}
\end{array}
\]
such that the diagram
\[
\xymatrix{
	\GL_2(\Z)_w \ar[rr]^-{\sim}\ar[dr]_-{\det} & \ar@{}[d]|{\rotatebox{180}{\ensuremath{\circlearrowright}}} & \Z[aw]^\times \ar[dl]^-{N_{K/\Q}} \\
	& \{ \pm 1 \} &
}
\]
commutes.
Here the symbol $ N_{K/\Q} $ is the norm map $ N_{K/\Q} \colon K^\times \to \Q^\times $.
Thus we have a group isomorphism
\[
\SL_2(\Z)_w \cong \{ \veps \in \Z[aw]^\times \mid N_{K/\Q}(\veps) = 1 \}.
\]

The following is a fundamental property of $ S_w $.

\begin{proposition} \label{prop:tilde_1}
	For a real quadratic number $w$, we have $ \widetilde{1}(w) = 2 \log \veps_w $.
\end{proposition}

\begin{proof}
	Let 
	\[
	\veps_w = (x_0 + y_0)/2, \ 
	\sigma :=
	\begin{pmatrix} w' & w/(w'-w) \\ 1 & 1/(w'-w) \end{pmatrix} 
	\in \SL_2(\R).
	\]
	Then
	\[
	\gamma_w = 
	\begin{pmatrix} \frac{x_0 - by_0}{2} & -cy_0 \\ ay_0 & \frac{x_0 + by_0}{2} \end{pmatrix}, \ 
	\sigma^{-1} \gamma_w \sigma =
	\begin{pmatrix} \veps_w^{-1} & 0 \\ 0 & \veps_w \end{pmatrix}.
	\]
	
	Fix $ z_0 \in \bbH $ and let $ z_1 := \sigma^{-1} z_0 $.
	Since $ \sigma^* \eta_w = \eta_{\infty, 0} = - z^{-1} dz $ by Lemma \ref{lem:eta}, we have
	\begin{align*}
		\widetilde{1}(w) 
		= \int_{z_0}^{\gamma_w z_0} \eta_w 
		= \int_{\sigma z_1}^{\sigma \sigma^{-1} \gamma_w \sigma z_1} \eta_w 
		= \int_{z_1}^{\sigma^{-1} \gamma_w \sigma z_1} \sigma^* \eta_w 
		= - \int_{z_1}^{\veps_w^{-2} z_1} \frac{dz}{z}
		= 2 \log \veps_w.
	\end{align*}
\end{proof}


\section{A proof of our main theorem} \label{sec:conti}


Throughout this section, we fix a point $ z_0 \in \bbH $.
We start with some lemmas.

\begin{lemma} \label{lem:int_trans}
	For non-empty even words $W_1, \dots, W_k$, we have
	\[
	\widetilde{f}(W_1 \cdots W_k)
	= \sum_{1 \le i \le k} \int_{z_0}^{\gamma_{W_i}^{} z_0} f \eta_{W_i \cdots W_k W_1 \cdots W_{i-1}}^{}.
	\]
\end{lemma}

\begin{proof}
	We have
	\[
	\widetilde{f} (W_1 \cdots W_k) 
	= \int_{z_0}^{\gamma_{W_1}^{} \cdots \gamma_{W_k}^{} z_0} f \eta_{W_1 \cdots W_k}
	\]
	by Lemma \ref{lem:gamma_W} \ref{item:lem:gamma_W;1}.
	Thus we obtain
	\begin{align*}
	\widetilde{f} (W_1 \cdots W_k) 
	&= \sum_{1 \le i \le k} \int_{\gamma_{W_1}^{} \cdots \gamma_{W_{i-1}}^{} z_0}^{\gamma_{W_1}^{} \cdots \gamma_{W_{i}}^{} z_0} f \eta_{W_1 \cdots W_k} \\
	&= \sum_{1 \le i \le k} \int_{z_0}^{\gamma_{W_{i}}^{} z_0} f (\gamma_{W_{1}}^{} \cdots \gamma_{W_{i-1}})^* \eta_{W_1 \cdots W_k}.
	\end{align*}
	We have
	\[
	\widetilde{f} (W_1 \cdots W_k) 
	= \sum_{1 \le i \le k} \int_{z_0}^{\gamma_{W_i} z_0} f \eta_{W_i \cdots W_k W_1 \cdots W_{i-1}}.	
	\]
	by Lemma \ref{lem:eta} and the fact that
	\[
	(\gamma_{W_{1}}^{} \cdots \gamma_{W_{i-1}})^{-1} [\overline{W_1 \cdots W_k}] = [\overline{W_i \cdots W_k W_1 \cdots W_{i-1}}],
	\]
\end{proof}

\begin{lemma} \label{lem:abs_eval}
	Let $ K $ be a compact subset in $ \bbH $ and
	$ y_0 := \min \{ \ImNew(z) \mid z \in K \} $.
	For real quadratic numbers $ v $ and $ w $, suppose the continued fraction expansions of $ v $ and $ w $ share first $ n $ terms.
	Then we have
	\[
	\left| \frac{1}{z - v} - \frac{1}{z - w} \right| \le \frac{2^{2-n}}{y_0^2}
	\]
	for any $ z \in K $.
\end{lemma}

\begin{proof}
	Fix $ z \in K $.
	Since $ |z-a| \ge \ImNew(z) \ge y_0 $ for any $ a \in \R $, we have
	\[
	\left| \frac{1}{z - v} - \frac{1}{z - w} \right|
	\le \frac{|v-w|}{|z - v| |z - w|}
	\le \frac{|v-w|}{y_0^2}.
	\]
	Also we get $ |v-w| \le 2^{2-n} $ by Lemma 1.24 in \cite{Aig}.
\end{proof}

\begin{lemma} \label{lem:Cauchy_seq}
	Let $W$ and $ W_1, \dots, W_n, \dots $ be non-empty even words such that $ W_n $ and $ W_{n+1} $ coincide in the first $ n $ terms and the last $ n $ terms.
	Let $\{ a_n \}_{n=1}^{\infty} $ be a sequence in $ \Z_{\ge 0} $ diverging to positive infinity.
	Set $ w := [\overline{W}] $.
	Then
	\[
	\left\{ \int_{z_0}^{\gamma_W^{} z_0} f \eta_{W_n}^{} \right\}_{n \ge 1}, \  
	\left\{ \int_{z_0}^{\gamma_W^{a_n} z_0} f \eta_{W^{a_n}W_n}^{} - a_n N(W) \widetilde{f} (w) \right\}_{n \ge 1}
	\]
	are Cauchy sequences.
	In particular, these sequences are bounded.
\end{lemma}

\begin{proof}
	Let
	\[
	A_n := \int_{z_0}^{\gamma_W^{} z_0} f \eta_{W_n}^{}, \  
	B_n := \int_{z_0}^{\gamma_W^{a_n} z_0} f \eta_{W^{a_n}W_n}^{} - a_n N(W) \widetilde{f} (w).
	\]
	Fix positive integers $ m \le n $ such that $ a_m \le a_n $.
	By Lemma \ref{lem:tilde_f} and Lemma \ref{lem:int_trans}, we have
	\[
	B_n - B_m
	= \int_{z_0}^{\gamma_W^{} z_0} f \left( 
	\sum_{1 \le i \le a_n} \eta_{W^{i} W_n W^{a_n-i}}^{}
	- \sum_{1 \le i \le a_m} \eta_{W^{i} W_m W^{a_m-i}}^{}
	- (a_n - a_m) \eta_{W}^{} \right).
	\]
	This is equal to
	\[
	- \int_{z_0}^{\gamma_W^{} z_0} f \left( 
	S_1(m,n,z) + S_2(m,n,z) - S_1'(m,n,z) - S_2'(m,n,z)
	\right) dz
	\]
	where
	\begin{align*}
	S_1(m,n,z) &:= \sum_{1 \le i \le a_m} \left(
	\frac{1}{z - [\overline{W^{i} W_n W^{a_n-i}}]} - \frac{1}{z - [\overline{W^{i} W_m W^{a_m-i}}]}
	\right), \\
	S_1'(m,n,z) &:= \sum_{1 \le i \le a_m} \left(
	\frac{1}{z - [\overline{W^{a_n-i} W_n W^{i}}]'} - \frac{1}{z - [\overline{W^{a_m-i} W_m W^{i}}]'}
	\right), \\
	S_2(m,n,z) &:= \sum_{a_m < i \le a_n} \left(
	\frac{1}{z - [\overline{W^{i} W_n W^{a_n-i}}]} - \frac{1}{z - w}
	\right), \\
	S_2'(m,n,z) &:= \sum_{a_m < i \le a_n} \left(
	\frac{1}{z - [\overline{W^{a_n-i} W_n W^{i}}]'} - \frac{1}{z - w'}
	\right). 
	\end{align*}
	
	We evaluate these four sums.
	Let $ y_0 := \min\{ \ImNew(z_0), \ImNew(\gamma_W^{} z_0) \} $.
	
	The continued fraction expansions of $ [\overline{W^{i} W_n W^{a_n-i}}] $ and $ [\overline{W^{i} W_m W^{a_m-i}}] $ share at least first $ i+m $ terms since $ W_m $ and $ W_n $ share first $ m $ terms by assumption.
	Thus we have
	\[
	| S_1(m,n,z) |
	\le \sum_{1 \le i \le a_m} \frac{2^{2-i-m}}{y_0^2} 
	= \frac{2^{2-m} - 2^{2-m-a_m}}{y_0^2}
	\]
	by Lemma \ref{lem:abs_eval}.
	
	By Lemma 1.28 in \cite{Aig}, the continued fraction expansions of $ -[\overline{W^{a_n-i} W_n W^{i}}]' $ and $ -[\overline{W^{a_m-i} W_m W^{i}}]' $ share at least first $ i+m $ terms since $ W_m $ and $ W_n $ share last $ m $ terms by the assumption.
	Thus we obtain 
	\[
	| S_1'(m,n,z) |
	\le \frac{2^{2-m} - 2^{2-m-a_m}}{y_0^2}
	\]
	by the same argument in Lemma \ref{lem:abs_eval}.
	
	Also we have
	\[
	| S_2(m,n,z) |
	\le \frac{2^{2-a_n} - 2^{2-a_m}}{y_0^2}, \ 
	| S_2'(m,n,z) |
	\le \frac{2^{2-a_n} - 2^{2-a_m}}{y_0^2}.
	\]
	
	Let $ C $ be the maximal value of $ f $ in the contour for the integration.
	As seen in the discussion above, we obtain
	\[
	|B_n - B_m|
	\le 2C |\gamma_W^{} z_0 - z_0| \frac{2^{2-m} - 2^{2-m-a_m} + 2^{2-a_n} - 2^{2-a_m}}{y_0^2}.
	\]
	Since the right hand side converges to 0 as $ m, n \to \infty $, $ \{B_n\}_{n \ge 1} $ is a Cauchy sequence.
	
	We can prove similarly that $ \{A_n\}_{n \ge 1} $ is a Cauchy sequence.
\end{proof}

%

\begin{proof}[Proof of Theorem \ref{thm:main}.]
	Keep the notation in Theorem \ref{thm:main} and let 
	\[
	W^{(n)} := V_0 W_1^{a_n^{(1)}} V_1 W_2^{a_n^{(2)}} V_2 \cdots W_k^{a_n^{(k)}} V_k.
	\]
	Then we have $ w^{(n)} = [\overline{W^{(n)}}] $.
	It follows from Lemma \ref{lem:int_trans} that
	\begin{align*}
		&N(W^{(n)}) \widetilde{f}(w^{(n)}) \\
		&= \sum_{0 \le i \le k} \int_{z_0}^{\gamma_{V_i}^{} z_0} f \eta_{V_i W_{i+1}^{a_n^{(i+1)}} \cdots V_k V_0 W_1^{a_n^{(1)}} V_1 \cdots W_i^{a_n^{(i)}}}^{} 
		+ \sum_{1 \le i \le k} \int_{z_0}^{\gamma_{W_i}^{a_n^{(i)}} z_0} f \eta_{ W_{i}^{a_n^{(i)}}V_i \cdots V_k V_0 W_1^{a_n^{(1)}} V_1 \cdots V_{i-1}}^{}.
	\end{align*}
	By Lemma \ref{lem:Cauchy_seq}, this can be written as
	\[
	\sum_{1 \le i \le k} a_n^{(i)} N(W_i) \widetilde{f}(w_i) + O(1)
	\]
	as $ n \to \infty $.
	Similarly we have
	\[
	N(W^{(n)}) \widetilde{1}(w^{(n)})
	= \sum_{1 \le i \le k} a_n^{(i)} N(W_i) \widetilde{1}(w_i) + O(1)
	\]
	by substituting $ f=1 $.
	
	Putting everything together we have
	\begin{align*}
	\lim_{n \to \infty} f(w^{(n)}) 
	&= \lim_{n \to \infty} \frac{N(W^{(n)}) \widetilde{f}(w^{(n)})}{N(W^{(n)}) \widetilde{1}(w^{(n)})} \\
	&= \lim_{n \to \infty}\frac{a_n^{(1)} N(W_1) \widetilde{f} (w_1) + \cdots + a_n^{(k)} N(W_k) \widetilde{f} (w_k) + O(1)}
	{a_n^{(1)} N(W_1) \widetilde{1} (w_1) + \cdots + a_n^{(k)} N(W_k) \widetilde{1} (w_k) + O(1)} \\
	&= \frac{a^{(1)} N(W_1) \widetilde{f} (w_1) + \cdots + a^{(k)} N(W_k) \widetilde{f} (w_k)}
	{a^{(1)} N(W_1) \widetilde{1} (w_1) + \cdots + a^{(k)} N(W_k) \widetilde{1} (w_k)}. 
	\end{align*}
\end{proof}

\section*{Acknowledgements}
I would like to show my greatest appreciation to Professor Takuya Yamauchi for introducing me to this topic and giving many advices. I am deeply grateful to Dr.~Toshiki Matsusaka for many discussion and valuable comments.

\bibliographystyle{plain}
\bibliography{myrefs_val}

\end{document}